\documentclass[10pt]{amsart}
\usepackage{amsthm}
\theoremstyle{definition}
\usepackage{amssymb}
\usepackage{amsmath}
\usepackage[all]{xy}
\usepackage{listings}
\usepackage{fullpage}
\usepackage{comment}
\usepackage{epsfig} 
\usepackage{epstopdf}
\usepackage{float}
\usepackage{soul}
\usepackage{epigraph}
\usepackage{caption}
\usepackage{subcaption}\usepackage{tikz-cd}
\usepackage[retainorgcmds]{IEEEtrantools}
\newtheorem{example}{Example}[section]
 
\newtheorem{lemma}[equation]{Lemma}
\newtheorem{theorem}[equation]{Theorem}
\newtheorem*{theorem*}{Theorem A}

\newtheorem*{corollary*}{Corollary}

\theoremstyle{definition}
\newtheorem{definition}[equation]{Definition}

\newtheorem{remark}[equation]{Remark}

\numberwithin{equation}{section}

\usepackage{tikz-cd}

\newcommand{\xiv}{\boldsymbol{\xi}}
\newcommand{\etav}{\boldsymbol{\eta}}

\newcommand{\N}{\mathbb{N}}
\newcommand{\R}{\mathbb{R}}

\usepackage{placeins}

\begin{document}
\title{Runge-Kutta  and Networks}
\author{Lee DeVille, Eugene Lerman and James Schmidt}
\address{Department of Mathematics, University of Illinois, Urbana, IL 61801}

\begin{abstract}
We categorify the RK family of numerical integration methods (explicit
and implicit).  Namely
we prove that if a pair of ODEs are related by an affine  map then the
corresponding discrete time dynamical systems are also related by the
map.  We show that in practice this works well when the pairs of
related ODEs come from the coupled cell networks formalism and, more
generally, from fibrations of networks of manifolds.
\end{abstract}
\maketitle

\epigraph{In theory there is no difference between theory and
  practice. In practice, there is.}{Attributed to various people}
\section{Introduction}
The goal of the paper is to study the compatibility of the RK family
of 
numerical integration methods with maps of dynamical systems.
Our initial motivation was to understand why the fourth order explicit
Runge-Kutta method (RK4) preserves polydiagonals in coupled cell
networks even when these invariant subsystems (the polydiagonals) are
exponentially unstable.

Coupled cell networks, which is an interesting class of continuous
time dynamical systems, were introduced by Golubitsky, Stewart, Pivato
and T\"or\"ok \cite{Golubitsky.Stewart.Torok.05,
  Stewart.Golubitsky.Pivato.03}.  They have been intensely studied by
many mathematicians for a number of years.  The polydiagonals of
coupled cell networks are invariant subsystems that ultimately arise
from the combinatorics of the networks in question.  The framework of
coupled cell networks has been generalized by the two of us \cite{DL1,
  DL2}.  We showed that the combinatorics of the networks leads not
only to invariant subsystems but more generally to maps between
dynamical systems.  Recall that maps between continuous time dynamical
systems are maps between their phase spaces that send the trajectories
of the first system to the trajectories of the second (see
Definition~\ref{def:related} and subsequent remarks).  Again we
could see in examples that maps of continuous time dynamical systems
were compatible with the discretizations provided by RK4.

The formalism of \cite{DL2} has been generalized further to networks
of open systems, see \cite{L}.    In particular, while the formalism
\cite{DL2}  produces maps of dynamical systems that are essentially
linear the formalism of \cite{L} can produce pairs of dynamical
systems related by truly nonlinear maps.
The main theoretical result of the paper can be now formulated as
follows. See Theorems~\ref{thm:1} and~\ref{thm:2} below for more precise
formulations.\\[4pt]

\noindent{\bf Theorem. }\quad Let $X:\R^n\to \R^n$ and $Y:\R^m\to \R^m$
be a pair of maps defining the ODEs $\dot{x} = X(x)$ and $\dot{y} =
Y(y)$.  Let $L:\R^n\to \R^m$ be a linear map and $p\in \R^m$ a
vector.  Suppose
\[
Y(Lx +p) = L(X(x)) 
  \]
for all $x\in \R^n$.   Let $D_X:\R^n\to \R^n$, $D_Y:\R^m \to \R^m$
denote a pair of discrete time dynamical system produced by a
Runge-Kutta method (explicit or implicit).   Then
\[
D_yY\, (Lx +p ) = L(D_X(x)) + p
  \]
  for all $x\in \R^n$.\\

Equivalently, if  a map of
dynamical systems $f: (X, \R^n) \to (Y, \R^m)$ is affine (i.e., $f(x)
= Lx +p$ for some linear map $L:\R^n\to \R^m$) then $f$ sends a
trajectory $\{x_n = (D_X)^n (x_0)\}_{n=1}^\infty$ of the discretized system $D_X$ to
the trajectory $\{y_n = (D_Y)^n (f(x_0))\}_{n=1}^\infty$ of the
discretized system $D_Y$.

In the case of coupled cell networks $\R^n = \R^{n_1}\times
\R^{n_2}\times \cdots \times \R^{n_r} $, $\R^m =
(\R^{n_1})^{k_1}\times (\R^{n_2})^{k_2}\times \cdots \times
(\R^{n_r})^{k_r}$ for some $n_1, \ldots, n_r$, $k_1, \ldots, k_r$ and
the map
\[
f:\R^{n_1}\times \R^{n_2}\times\cdots \times \R^{n_r} \to (\R^{n_1})^{k_1}\times (\R^{n_2})^{k_2}\times \cdots \times (\R^{n_r})^{k_r}
  \]
  is of the form
 \[ 
f(x_1,\ldots x_r) = (\underbrace{(x_1,\ldots, x_1)}_{k_1 }, \ldots, \underbrace{(x_r,\ldots, x_r)}_{k_r}).
\]
Thus our theorem proves that polydiagonals in coupled cell networks
are preserved by any numerical method in the RK family.  We give an
example to show that this works in practice even if the polydiagonal
(that is, the invariant subsystem $f(\R^n)$) is exponentially
unstable.  We admit that at the first glance this may seem ``obvious"
for explicit methods given the form of the map $f$.  After all, $f$ is
just duplicating certain groups of coordinates.  We hope that upon
further reflection the reader will see that this is not completely
obvious even in the case of explicity methods.  Recall that an explict
RK method require composing two or more nonlinear maps, taking a
linear combinations of the composites, composing again and so on.  It
requires an argument why a repeated application of these operations
preserves duplication of coordinates.

\subsection*{Organization of the paper}  In section~\ref{sec:2}  we recall
some of the relevant background material.   In section~\ref{sec:3} we
prove our main theorem for explict RK methods.    In
section~\ref{sec:4} we extend the result to implicit RK
methods. Section~\ref{sec:5} is taken up with examples.  There we show
that explicity fourth order Runge-Kutta (RK4) works well for preserving
polydiagonals and affine polydiagonals.  We then  illustrate a difference between theory and practice by an example
of a linear map of dynamical systems that {\em in practice} does not
preserve the discretizations.    The issue is likely to be  the roundoff
errors.   Finally we give an example of an invariant parabola in
$\R^2$ which is not preserved by RK4.     We agree that this should
not be surprising since non-geometric numerical methods are not known
for their ability to preserve nonlinear invariant submanifolds.

\section{Background}\label{sec:2}
We start with the key definition, which is %
standard in differential
geometry.
\begin{definition} \label{def:related} 
Let $f:N\to M$ be a differentiable map between two manifolds.  A
vector field $X $ on $N$ is {\em $f$-related} to the vector field $Y$
on $M$ if
\begin{equation} \label{eq:1.2}
df_x( X(x)) = Y(f(x))
\end{equation}
for all $x\in N$.  Here and elsewhere $df_x:T_xN\to T_{f(x)}M$ denotes the
differential of the map $f$.   
  \end{definition}
  \begin{remark} \label{rmrk:related}
In the case where $N= \R^n$ and $M= \R^m$ the vector field $X$ on $N$
is usually identified with a map $X:\R^n\to \R^n$, and similarly $Y$ is
identified with a map $Y:\R^m\to \R^m$.  The equation \eqref{eq:1.2}
then 
reduces to
\[
df(x) \, X(x) = Y (f(x))
  \]
where $df(x)$ is the Jacobian matrix of the map $f$.
\end{remark}

\begin{remark}
A simple application of the uniqueness of solutions of ODEs and of the
chain rule shows that if a vector field $X$ on a manifold $N$ is
$f:N\to M$ related to a vector field $Y$ on a manifold $M$ then for
any integral curve $\gamma(t)$ of $X$, $f(\gamma(t)) $ is an integral
curve of $Y$.  See for example \cite{Warner}.
\end{remark}

\begin{remark}
It is common to refer to a pair $(N,X)$ where $N$ is a manifold and
$X$ is a vector field on $N$ as a {\em  continuous time dynamical
  system}.   A {\em map of dynamical systems} from a system $(N,X)$ to
a system $(M,Y)$ is a differntiable map $f:N\to M$ so that $X$ is
$f$-related to $Y$.  Continuous time dynamical systems and  their maps
form a category, see for example \cite{L}.   One may may interpret the
main results of the paper as an attempt to construct a class of functors from 
a category of continuous time systems to the category of discrete
time systems using numerical  integration methods.  The attempt 
succeeds in the case where the objects of the source category are
Euclidean (i.e., coordinate) vector spaces, the morphisms are affine
maps and the functors are constructed using the RK integration methods.
 \end{remark} 
 We now turn to numerical methods.

\begin{definition}\label{def:rk}
  Consider a vector field $X: \R^n\to\R^n$.  Choose $s\in \N$,
  $\{a_{ij}\}_{i,j=1}^s$, and $\{b_i\}_{i=1}^s$.  The {\bf Runge-Kutta method with
    matrix $a_{ij}$, weights $b_i$, and stepsize $h$}, denoted
  $D_X^{(A,b,h)}$, is defined as follows (see ~\cite[(12.51),
  pages 351---352]{Suli.Mayers.book}):  For $i=1,\dots, s$, we
  set
  \begin{equation*}
    k_{X,i}(x) := X\left(x + h \sum_{j=1}^s a_{i,j} k_{X,j}(x)\right),
  \end{equation*}
  and then 
  \begin{equation*}
    D_X^{(A,b,h)}(x) := x + h \sum_{i=1}^s b_i k_{X,i}(x).
  \end{equation*}  
We say that the method is {\bf explicit} if $a_{ij} = 0$ whenever
$i\le j$.  Otherwise the method is {\bf implicit}.  
\end{definition}

\begin{remark}
In an explicit RK method for a vector field $X$
\begin{eqnarray*}
k_{X,1}(x) &:=& X(x),\\
  k_{X,2}(x)&:= &X(x + h  a_{2,1} k_{X,1}(x)),\\
  k_{X,3}(x)&:= &X(x + h ( a_{3,1} k_{X,1}(x)) +  a_{3,2}
                  k_{X,2}(x)))),\\
  \vdots& &\\
   k_{X,i}(x) & := &X(x + h ( a_{i,1} k_{X,1}(x)+\cdots+a_{i,i-1}
  k_{X,i-1}(x) ))   \\
  \vdots &&
 \end{eqnarray*} 
 \end{remark} 

Once a Runge-Kutta method is fixed, we {\em numerically integrate} the
ODE
\begin{equation} \label{eq:ODE}
  \dot{x} = X(x), x(0) = x_0
\end{equation}
  by the following iterative scheme:
\begin{equation*}
  x_{n+1} = D_X^{(A,b,h)}(x_n), \qquad \textrm{ for all } n\geq 0.
\end{equation*}
There is a large theory of the accuracy, efficiency, convergence, and
consistency of such methods, which we do not address here.   See, for example,~\cite{Suli.Mayers.book, Iserles.book}.   
Under certain well-understood conditions, the solution (integral
curve) $x(t)$ of \eqref{eq:ODE} will be well-approximated by
$x_{\lfloor(t/h)\rfloor}$ for all $t\geq 0$.

\begin{example}  The fourth order Runge-Kutta method (RK4) is defined
  by the following data:
\begin{equation*}
  A = \left(\begin{array}{cccc}0&0&0&0\\1/2&0&0&0\\0&1/2&0&0\\0&0&1&0\end{array}\right), \quad b = (1/6,1/3,1/3,1/6)^T.
\end{equation*}
Note that RK4 is explicit.
\end{example}

  
\section{Explicit RK methods}\label{sec:3}
In this section we prove

\begin{theorem}\label{thm:explicit}\label{thm:1}
  Let $X:\R^n\to\R^n$ and $Y:\R^m\to\R^m$ be two vector fields,  $L: \R^n\to\R^m$ a linear map and $p\in \R^m$, such that
\begin{equation}\label{eq:AffineAssumption}
  LX(x) = Y(Lx+p) \qquad \textrm{ for all }x\in \R^n.
\end{equation}
Then for any choice of a vector $b$ and a matrix $(a_{i,j})$ in
Definition~\ref{def:rk} that gives an explicit RK method (i.e.,
$a_{ij} = 0 $ for $i\leq j$), we have
\begin{equation}\label{eq:AffineConclusion}
  LD_X^{(A,b,h)}(x) + p = D_Y^{(A,b,h)}(L x + p).
\end{equation}
\end{theorem}

\begin{remark}
Let $X$, $Y$, $f(x) = Lx + p$, $D_X^{(A,b,h)}$ and $ D_Y^{(A,b,h)}$ be
as above.  Pick $x_0 \in \R^n$.  Set $y_0 := f(x_0) = Lx_0 +p$.
Define recursively
\[
x_{n+1}: = D X^{(A,b,h)} (x_n)\qquad \textrm{and}\qquad y_{n+1}: = D Y^{(A,b,h)} (y_n).
\]
An induction argument based on  Theorem~\ref{thm:1} implies that
  \[
y_n = f(x_n) 
\]
for all $n$.
 \end{remark} 
 Our proof of Theorem~\ref{thm:1} is based on a lemma.

 \begin{lemma}\label{lem:fkgk}
   Fix $q\geq 0$.   Let $f:\R^n\to \R^m$, $f(x) = Lx + p$ be an affine
   map as in Theorem~\ref{thm:1}.   
  Let $f_0, f_1,\dots, f_q:\R^n\to\R^n$ and $g_0,
  g_1,\dots,g_q:\R^m\to\R^m$ be two collections of maps where 
\begin{equation*}
  g_k ( L x + p) = Lf_k(x) \qquad \textrm{for all } 0\leq k \leq q.
\end{equation*}
Fix $\gamma_1,\dots, \gamma_q\in \R$ and  define the functions $\varphi:\R^n\to\R^n$
and $\psi: \R^m\to\R^m$ by 
\begin{align*}
  \varphi(x) &:= f_0\left(x + \sum_{i=1}^q \gamma_i f_i(x)\right),\\
  \psi (y) &:= g_0\left(y + \sum_{i=1}^q \gamma_i g_i(y)\right).
\end{align*}
 Then
\begin{equation*}
 \psi (L x + p ) = L\varphi (x).
\end{equation*}
\end{lemma}

\begin{proof}
 The proof is a computation:
\begin{align*}
  \psi (L x + p)
  	&= g_0\left(L x + p + \sum_{i=1}^q \gamma_i g_i(L x + p)\right)\\
	 & = g_0\left(L x + p+ \sum_{i=1}^q \gamma_i L f_i(x)\right)\\
	&= g_0\left(L \left(x  + \sum_{i=1}^q \gamma_i  f_i(x)\right)+ p\right)\\
	&= L f_0 \left(x  + \sum_{i=1}^q \gamma_i  f_i(x)\right) =
  L\varphi (x).\\
\end{align*}
\end{proof}
\begin{proof}[Proof of Theorem~\ref{thm:1}]
  We first prove by induction that
\begin{equation}\label{eq:k}
  k_{Y,i}(Lx +p) = Lk_{X,i}(x)
\end{equation}
  for all $i=1,\dots, s$.   Since $k_{X,1}(x)  = X(x)$ and $k_{Y,1}(x)
  = Y(x)$, the equation  \eqref{eq:k} holds for $i=1$ by assumption:
  see \eqref{eq:AffineAssumption}.  Now assume that~\eqref{eq:k} holds
  for $j=1,\dots, i-1$.  Then we compute:
  \begin{align*}
    k_{Y,i}(Lx+p)
    &= Y\left((Lx + p) +  h\sum_{j=1}^{i-1} a_{i,j}k_{Y,j}(Lx+p)\right) & \\
    &= Y\left(Lx + p + h\sum_{j=1}^{i-1} a_{i,j}L\big(k_{X,j}(x)\big)
      \right)
    & \textrm{  by inductive assumption} \\
    &= Y\left(L\left(x + h\sum_{j=1}^{i-1}
      a_{i,j}k_{X,j}(x)\right)+ p \right) &  \textrm{ by Lemma~\ref{lem:fkgk}}\\
    &= LX\left(x + h\sum_{j=1}^{i-1} a_{i,j}k_{X,j}(x)\right) & = L k_{X,i}(x). 
\end{align*}
Finally
\begin{align*}
  D_Y^{(A,b,h)}(L x + p)
  	&= (Lx+p) + h\sum_{j=1}^s b_j k_{Y,j}(Lx+p) 
 	= (Lx+p) + h\sum_{j=1}^s b_j L\big(k_{X,j}(x)\big)\\ 
	&= L\left(x + h\sum_{j=1}^s b_j k_{X,j}(x)\right) + p
	= LD_X^{(A,b)}(x) + p,
\end{align*}
and we are done.
\end{proof}    

\section{Implicit RK methods} \label{sec:4}

We now consider implicit Runge--Kutta methods.  Recall the defining
equation for the functions $k_{X,i}$ in Definition~\ref{def:rk}: 
\begin{equation} \label{eq:k2}
    k_{X,i}(x) = X\left(x + h \sum_{j=1}^s a_{i,j} k_{X,j}(x)\right).
  \end{equation}
  Unless the numbers $a_{i,j} =0$ for all $i\leq j$ \eqref{eq:k} is a
  system of nonlinear ``algebraic'' equations.  In general there no
  easy way to solve this system of equations and obtain a formula for
  $k_{X,i}(x)$ in terms of $k_{X,1}$, \ldots, $k_{X,i-1}$.  One
  solution to the problem is  to choose a small enough step $h$
  so that the contraction mapping principle applies to the
  appropriately defined map.  Then one chooses
  a starting point and iterates.  See for example
  \cite[Chapter 6]{Iserles.book}.

\begin{definition}\label{def:irki}
  We define an {\em implicit Runge-Kutta method with $q$-step
    iterative solution} as follows:  Choose $s,A,b,h$ as in
  Definition~\ref{def:rk}. Choose a positive integer $q$ and fix a
  point $x\in \R^n$.  Define a
  map
\[
  \underline{X}:  (\R^n)^s
  \to (\R^n)^s
\]
by 
  \begin{equation}\label{eq:4.3}
     \underline{X}(\xi_1,\ldots, \xi_s)= \left(X\left(x + h \sum_{j=1}^s
         a_{1,j} \xi_j\right), \ldots,  X\left(x+ h \sum_{j=1}^s a_{s,j} \xi_j\right)
\right)
\end{equation}
for all $\xiv=(\xi_1,\ldots, \xi_s) \in (\R^n)^s$.
  Choose $\xiv^{(0)} =
  (X(x),\ldots, X(x))$.   Define $\xiv^{(k+1)} \in (\R^n)^s$ recursively by
  $\xiv^{(k+1)}:  = \underline{X}(\xiv^{(k)})$ for $k=0,\ldots,
  q-1$. Now  define
 \begin{equation*}
    D_X^{(A,b,h,q)}(x) := x + h \sum_{i=1}^s b_i \xi_i^{(q)}.
  \end{equation*}    
\end{definition}
\noindent
We are now in position to state the second main result of the paper.

\begin{theorem}\label{thm:implicit}\label{thm:2}
 Let  $L: \R^n\to\R^m$ a linear map, $p\in \R^m$ a point and  let $X:\R^n\to\R^n$ and $Y:\R^m\to\R^m$ be two vector fields such that
\begin{equation}\tag{\ref{eq:AffineAssumption}}
  LX(x) = Y(Lx+p) \qquad \textrm{ for all }x\in \R^n.
\end{equation}
Define $D_X^{(A,b,h,q)}(x)$, $D_Y^{(A,b,h,q)}(y)$ as above in
Definition~\ref{def:irki}.  Then for any choice of the parameters
$A,b,h$ and $q$ and for any $x\in \R^n$
\begin{equation}\tag{\ref{eq:AffineConclusion}}
  LD_X^{(A,b,h,q)}(x) + p = D_Y^{(A,b,h,q)}(L x + p).
\end{equation}
\end{theorem}
\mbox{}
\begin{proof}
By definition $  D_X^{(A,b,h,q)}(x) := x + h \sum_{i=1}^s b_i
\xi_i^{(q)}$ where $\xiv^{(q)}$ is defined recursively by
\[
\xiv^{(0)} =
  (X(x),\ldots, X(x))\qquad \textrm{and}\qquad \xiv^{(k+1)}:  = \underline{X}(\xiv^{(k)})
\textrm{ for } k=0,\ldots,
q-1;
\]
and $\underline{X}: (\R^n)^s \to (\R^n)^s$ is given by
\eqref{eq:4.3}.  Define
$\underline{Y}:(\R^m)^s \to (\R^n)^s$ by changing what
needs to be changed in \eqref{eq:4.3}.   In particular $x$ is replaced
by $Lx +p$. Let
\[
\etav^{(0)} = (Y(Lx+p),\ldots, Y(Lx+p)) = (\overbrace{L\times \cdots
  \times L}^{s} )\,(\xiv^{(0)}) \in (\R^m)^s
\]
and define $\etav^{(q)}$ recursively by
\[
\etav^{(k+1)}:  = \underline{Y}(\etav^{(k)})\quad \textrm{ for } k=0,\ldots,
q-1.
\]
It is easy to show by induction that
\[
\etav^{(k)}  =  (\overbrace{L\times \cdots \times L}^{s} )\,(\xiv^{(k)}) 
  \]
  for $k=1,\ldots, q$.   Here is a proof of  the inductive step:
 
\begin{align*}
  \etav^{(k+1)}
  	&= \left(Y\left(y + h \sum_{j=1}^s a_{1,j} \eta^{(k)}_j\right),
     \ldots, Y\left(y + h \sum_{j=1}^s a_{s,j} \eta^{(k)}_j\right)\right)\\
	&= \left(Y\left(Lx +p + h \sum_{j=1}^s a_{1,j} L\xi^{(k)}_j\right),
     \ldots, Y\left(Lx+p + h \sum_{j=1}^s a_{s,j} L\xi^{(k)}_j\right)\right)\\
&= \left(Y\left(L\left(x + h \sum_{j=1}^s a_{1,j} \xi^{(k)}_j \right) +p \right),
     \ldots, Y\left(L\left(x+ h \sum_{j=1}^s a_{s,j} \xi^{(k)}_j\right) +p
                                                                          \right)\right)\\
 &= \left(LX\left( x + h \sum_{j=1}^s a_{1,j} \xi^{(k)}_j \right),
     \ldots, Lx\left(L(x+ h \sum_{j=1}^s a_{s,j} \xi^{(k)}_j
   \right)\right)\\
        &	= \left( L\xi^{(k+1)}_1, \ldots, L\xi^{(k+1)}_s \right) = (\overbrace{L\times \cdots \times L}^{s} )\,(\xiv^{(k+1)}) .
\end{align*}
In particular
\[
\left( \eta^{(q)}_1, \ldots, \eta^{(q)}_s \right) = \left( L\xi^{(q)}_1, \ldots, L\xi^{(q)}_s \right) .
\]
Finally
\begin{align*}
  D_Y^{(A,b,h,q)}(L x + p)
  	&= (Lx + p) + h\sum_{i=1}^s b_{i}\eta_i^{(q)}
	= (Lx + p) + h\sum_{i=1}^s b_{i}L\xi_i^{(q)}\\
	&= L \left(x + h\sum_{i=1}^s b_{i}\xi_i^{(q)}\right) + p = LD_X^{(A,b,h,q)}(x) + p.
\end{align*}
 \end{proof}

\section{Examples} \label{sec:5}

We present several examples illustrating the results of the theorems
above.  In many of the examples listed below, we want to check whether
or not the vector fields are related in the sense defined in the
introduction.  As such, one of the quantities that we plot is the
scalar quantity $\left\|{D_Y \circ f - f\circ D_X}\right\|_1$, where
the subscript denotes that we are taking the $\ell^1$ norm of the
vector.  This quantity is identically zero if the vector fields $X$
and $Y$ are $f$-related.

\begin{example}\label{example:1}
Consider the vector field
\begin{equation}\label{system:1}
Y:\R^2 \to \R^2, \qquad Y
\begin{pmatrix}x_1\\x_2
\end{pmatrix}
=\begin{pmatrix} -x_1 -2x_2 +	(x_1-x_2)x_1^3\\ 
-2x_1 -x_2 
\end{pmatrix}. 
\end{equation}
It is easy to see that $\Delta= \{x_1 = x_2\}$ is an invariant
submanifold of the vector field $Y$, since $Y\begin{pmatrix} x \\
  x\end{pmatrix} = -3\begin{pmatrix} x \\ x\end{pmatrix}$.  The
linearization of $Y$ at the origin (in fact, along any point on the diagonal $\{x_1= x_2\}$) is the matrix 
$\begin{pmatrix} -1 & -2 \\
  -2 & -1\end{pmatrix}$.
It follows that the diagonal $\Delta$ is an unstable submanifold of $Y$.
Nonetheless, the diagonal   is preserved by numerical integration.

Now for any vector field $Y$ on a manifold $M$ and an invariant
submanifold $\Sigma$ of $Y$ the inclusion $f:\Sigma \hookrightarrow M$
relates the restriction $X = Y|_\Sigma$ and $Y$.  In the case of the
example before us, the inclusion $L:\Delta \to \R^2$ of the invariant
submanifold is $L(u) = (u,u)$ and $X(u) = -3u$.  Since the map $L$ is
linear Theorem~\ref{thm:1} applies.
Figure \ref{ex:fig1} graphically illustrates Theorem~\ref{thm:1} for this pair of systems: it shows that $L D_X^{(A,b,h)}(x) = D_Y^{(A,b,h)}(Lx)$, which by extension demonstrates the invariance of the diagonal.  The left figure shows a phase portrait of system \ref{system:1} while the right figure shows agreement of numerical integration. 

 \begin{figure}[ht]\centerline{
\includegraphics[scale=.55]{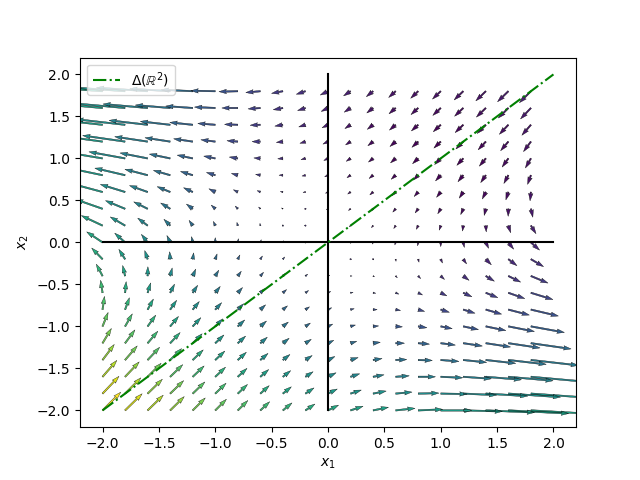}
\includegraphics[scale=.395]{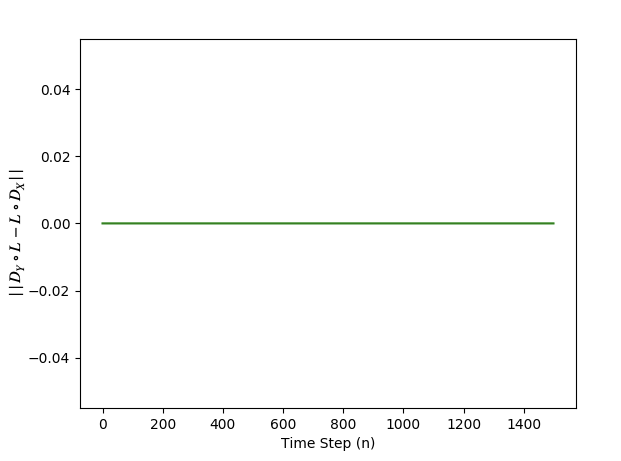}}\caption{Numerics for system \eqref{system:1}. The left figure shows a phase portrait of system \ref{system:1} while the right figure shows agreement of numerical integration.  }\label{ex:fig1}
\end{figure}
\end{example}

\newpage 
\begin{example} \label{example:2} This example shows that affine
  invariant submanifolds also are preserved in practice.  We consider
  the vector field
\begin{equation}\label{system:22}
Y:\R^2 \to \R^2, \qquad Y
\begin{pmatrix}x_1\\x_2
\end{pmatrix}
=\begin{pmatrix} -x_1 -2x_2 +	1\\ 
-2x_1 -x_2 
\end{pmatrix}. 
\end{equation}
The map 
\[
f:\R\to \R^2, \qquad f(u) = (u,u+1)
\]
is of the form $f(u) = Lu +(0,1)$ where $L:\R \to \R^2$ is the linear
map $L(u) = (u,u)$. The affine map $f$ 
relates the vector field $X(u) = -3u-1$ and $Y$:
\[
Y(f(u)) = Y(u,u+1) = (-3u-1, -3u -1) = Df (u) \,\,X(u).
\]
Thus the affine submanifold $\{x_2 = x_1 +1\}$ is an
invariant submanifold of $X$.  A simulation shows that Runge-Kutta preserves this ``offset'' diagonal. 

  \begin{figure}[ht]\centerline{
\includegraphics[scale=.5]{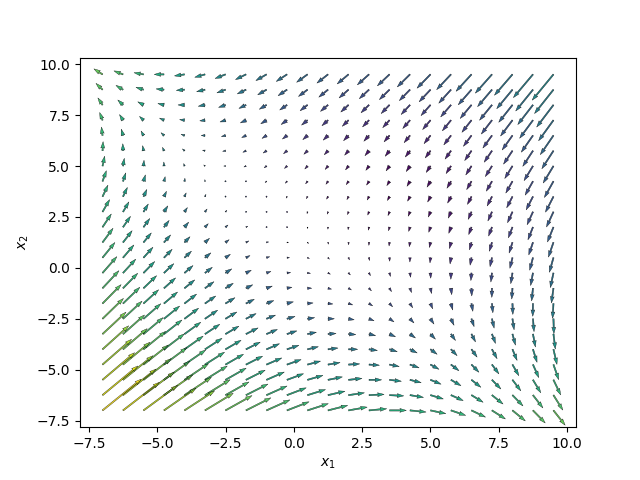}
\includegraphics[scale=.37]{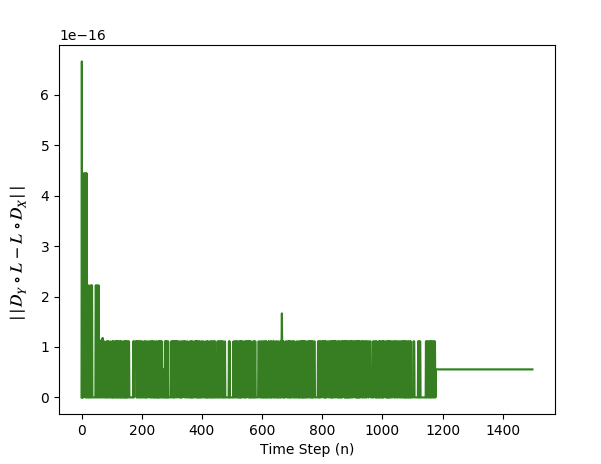}
}\caption{Numerics for system \ref{system:22}. The left figure shows a phase portrait of system \ref{system:1} while the right figure shows agreement of numerical integration.}
\end{figure}
\end{example}

\newcommand{\Pb}{\mathbb{P}}
\newpage
\begin{example}\label{example:gFib}
In this example we consider a pair of related vector fields produced
by the networks of manifolds formalism of \cite{DL2}.  Suppose we
choose any three functions $w_1:\R\to \R$, $w_2:\R^3 \to \R$ and $w_3:
\R^2 \to \R$.  Define a vector field $X:\R^3\to \R^3$ by
\[
X(x_1,x_2, x_3) = (w_1(x_1), w_2(x_2,x_1,x_1), w_3(x_3, x_2)).
  \]
  Define a vector field $Y:\R^3 \to \R^3$ by
\[
Y(y_1,y_2, y_3) = (w_1(y_1), w_1(y_2), w_2(y_3, y_1,  y_2)).
  \]
  The first vector field comes from the network
  \begin{equation}
  \xy
(-10,10)*{1\, \bullet}="1";
(-10,-10)*{2\, \bullet}="2"; 
(10,0)*{\bullet \,3}="3"; 
{\ar@/^.5pc/ "1";"3"};
{\ar@/_.5pc/"2";"3"};
\endxy
\end{equation}
and the second from the network
\begin{equation}
\xy
(-10,0)*{1\, \bullet}="1";
(10,0)*{ \bullet}="2"; (10,3)*+{2}="c";
(30,0)*{\bullet \,3}="3"; 
{\ar@/^1pc/ "1";"2"};
{\ar@/_1pc/ "1";"2"};
{\ar@{->} "2";"3"};
\endxy \quad .
\end{equation} There is a map of networs from the first to the second.  Out of
this map of networks the
machinary of \cite{DL2} produces the function
\[
f:\R^3\to \R^3,\qquad f(x_1,x_2, x_3) = (x_1,x_1, x_2)
  \]
  with the property that
  \begin{equation}\label{eq:5.3}
Y\circ f = df \circ X.
\end{equation}
It is also easy to  check directly that \eqref{eq:5.3} holds.
Figure~\ref{fig:3} shows that the $1$-norm of the difference $D_Y\circ
f - f\circ D_X$ stays identically zero throughout the simulation.  In this simulation, we take $w_1(x) = x$, $w_2(y,x_1,x_2) = \frac{\sin(x_1)\cdot x_2}{y}$, and $w_3(z,y) = y\cdot z$.

\begin{figure}[ht]
  \centerline{
\includegraphics[scale=.65]{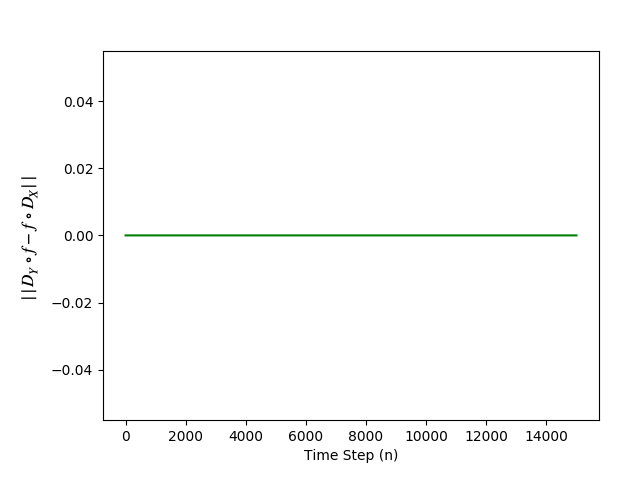}}\caption{Relatedness for example \ref{example:gFib}\label{fig:3}}
\end{figure}
\end{example}

\newpage
\begin{example}\label{example:4}
  Recall that in the coupled cell network formalism of Golubitsky,
  Stewart and their collaborators all the invariant submanifolds are
  vector subspaces and their inclusions are, of course, linear.  The
  formalism developed in \cite{DL1, DL2} is more general.  There the
  maps between dynamical systems are projections followed by diagonal
  embeddings.  In the case where the phase spaces are coordinate
  vector spaces all the maps are again linear.  Consequently as we
  proved in Theorem~\ref{thm:1} explicit RK methods work well for
  these types of networks.

  The approach of \cite{DL2} is generalized in \cite{L} in several
  directions.  In particular maps between dynamical systems
  constructed in \cite{L} need not be linear.  Consider the map
\[
f:\R\to \R^2, \qquad f(x) = (x^2, x).
\]
We claim that for any function $g:\R^2\to \R$ there are vector fields
$X:\R\to \R$ and $Y:\R^2\to \R^2$ which are $f$-related.   Indeed let
\[
X(x) = xg(x^2, x^2)
\]
and let
\[
Y(y_1,y_2) = (2y_1g(y_1,y_2^2), y_2g(y_2^2, y_1)).
\]
Then  
\[
Y (f(x)) = Y(x^2,x)=  (2x^2 g(x^2,x^2),  x g(x^2,x^2)), 
\]
\[
Tf_x = (2x,1)
\]
and
\[
Tf_x X(x) = (2x^2 g(x^2, x^2), xg(x^2, x^2)) = Y(f(x))
\]
for all $x\in \R$. It follows that  the parabola 
\[
P:= \{(x_1,x_2)\in \R^2 \mid x_2^2= x_1\}, 
\]
the image of $f$, is an invariant submanifold of the vector field $Y$.  We present two simulations demonstrating that the parabola is not preserved under numerics, the first where $g(x,y) = 0.999 + \sin(x+y) + \frac{1}{x+y}$, and the second where $g(x,y) \equiv 1$:

  \begin{figure}[ht]
     \centering
     \begin{subfigure}[b]{0.45\textwidth}
         \centering
        \includegraphics[width=\textwidth]{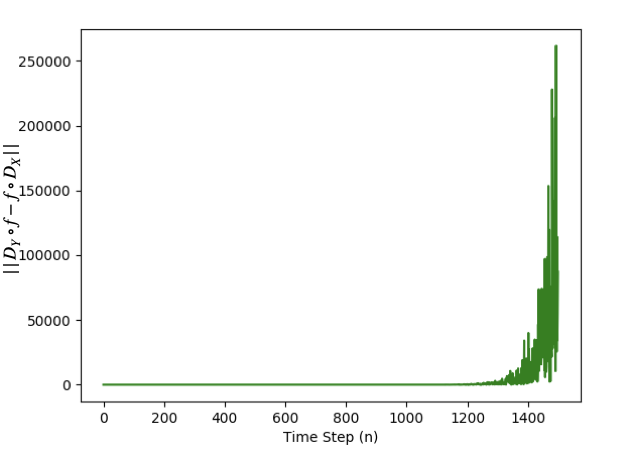}
         \caption{$g(x,y)= 0.999 + \sin(x+y) + \frac{1}{x+y}$}
         \label{fig:y equals x}
     \end{subfigure}
     \hfill
     \begin{subfigure}[b]{0.45\textwidth}
         \centering
         \includegraphics[width=\textwidth]{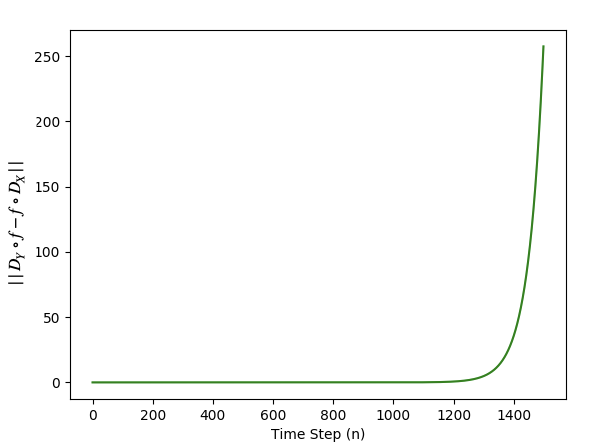}
         \caption{$g(x,y) \equiv 1$}
         \label{fig:three sin x}
     \end{subfigure}
           \caption{Non-invariance of parabola}
        \label{fig:three graphs}
\end{figure}
  
  \end{example}

\newpage
\begin{example}\label{ex:LinearMap}
Now we present an instance of theorem \ref{thm:1} with  a nontrivial linear map $A:\R^2\rightarrow\R^2$,  given by $A = \begin{pmatrix} 1 & -1 \\ 1 & 1\end{pmatrix}$.   Let $X = \begin{pmatrix} 0 & -1 \\ 1 & 0 \end{pmatrix}$ and $Y = \begin{pmatrix} 1 & 0 \\ 0 & -1\end{pmatrix}$ be linear vector fields on $\R^2$. A quick calculation shows that $Y \circ A = A\circ X$, and hence that $(X,Y)$ are $A$-related. Theorem \ref{thm:1} tells us that numerics should agree, and they would with infinite precision numerics.  However, while we see numerical agreement for many time steps,  it appears that  errors begin to arise after many more. 
\begin{figure}[ht]\label{fig:5}
  \centerline{
\includegraphics[scale=.4]{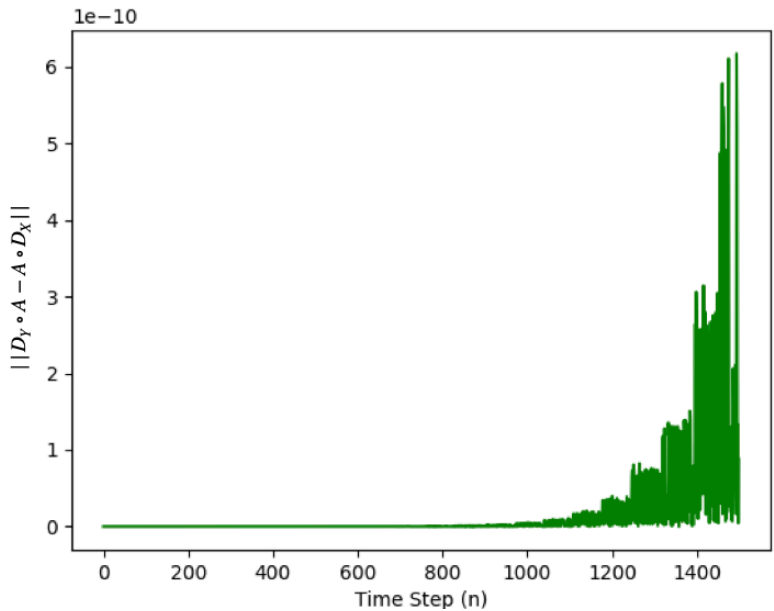}}\caption{Relatedness for example \ref{ex:LinearMap}}
\end{figure}	

\end{example}





\begin{thebibliography}{10}

\bibitem{Brown} H.D.\ Brown, {\em Near Algebras}, PhD thesis, Ohio
  State Univ., 1966.

\bibitem{DL1} L. DeVille and E. Lerman, Modular dynamical systems on
  networks, {\em J.\ Eur.\ Math.\ Soc.} 17 (2015), no. 12,
  2977--3013. {\tt arXiv:1303.3907 [math.DS]}
  
\bibitem{DL2} L. DeVille and E. Lerman, Dynamics on networks of
  manifolds, {\em SIGMA} {\bf 11} (2015), Paper 022, 21 pp.  {\tt
    arXiv:1208.1513 [math.DS]}
  
\bibitem{Golubitsky.Stewart.Torok.05}
M.\ Golubitsky, I.\ Stewart and A.\ T{\"o}r{\"o}k,
 Patterns of synchrony in coupled cell networks with multiple arrows.
 {\em SIAM J. Appl. Dyn. Syst.}, 4(1):78--100 (electronic), 2005.

 
\bibitem{Iserles.book}
A.\ Iserles, \emph{A first course in the numerical analysis of differential
  equations}, Cambridge Texts in Applied Mathematics, Cambridge University
  Press, Cambridge, 1996. \MR{1384977}

  
\bibitem{L} E. Lerman, Networks of open systems, {\em J.\ Geom.\
    Phys.} {\bf 130} (2018), 81--112; {\tt arXiv:1705.04814
    [math.OC]}

 \bibitem{MQ} R.I.\ McLachlan, R.I. and G.R.W.\ Quispel, Geometric
   integrators for ODEs. {\em Journal of Physics A: Mathematical and
     General} ,{\bf 39} (2006), p.5251-5285.


  
\bibitem
  {Stewart.Golubitsky.Pivato.03}
I.\ Stewart, M.\ Golubitsky  and M.\ Pivato, \emph{Symmetry groupoids and
  patterns of synchrony in coupled cell networks}, SIAM Journal on Applied
  Dynamical Systems \textbf{2} (2003), no.~4, 609--646.

\bibitem{Suli.Mayers.book}
Endre S\"{u}li and David~F. Mayers, \emph{An introduction to numerical
  analysis}, Cambridge University Press, Cambridge, 2003. \MR{2006500}

 \bibitem{Warner} F.W.\ Warner, {\em Foundations of differentiable
     manifolds and Lie groups}, Springer-Verlag, New York Berlin
   Heidelberg Tokyo, 1983.


\end{thebibliography}
\end{document}